\documentclass[10pt,a4paper]{amsart}
\usepackage{amssymb}%,showkeys
\newtheorem{theorem}{Theorem}[section]
\newtheorem{proposition}[theorem]{Proposition}
\newtheorem{lemma}[theorem]{Lemma}
\newtheorem{corollary}[theorem]{Corollary}

\theoremstyle{remark}

\newtheorem*{examples}{Examples}
\numberwithin{equation}{section}
\newcommand{\remove}[1]{ }
\newcommand{\set}[1]{\left\{#1\right\}}

\def\uu{\mathcal U}
\def\vv{\mathcal V}
\def\ww{\mathcal W}
\def\uuu{\overline{\mathcal U}}

\begin{document}
\title{On the number of unique expansions in non-integer bases}
\author{Martijn de Vries}
\address{Delft University of Technology, Mekelweg 4, 2628 CD Delft, the Netherlands}
\email{w.m.devries@ewi.tudelft.nl}
\thanks{The author has been supported by NWO, Project nr. ISK04G}
%\author{}
%\address{}
%\curraddr{}
%\email{}
%\subjclass{}

%\dedicatory{.}

%
\begin{abstract}
Let $q > 1$ be a real number and let $m=m(q)$ be the largest integer smaller than $q$. It is well known that each number
$x \in J_q:=[0, \sum_{i=1}^{\infty} m q^{-i}]$ can be written as $x=\sum_{i=1}^{\infty}{c_i}q^{-i}$ with integer coefficients
$0 \le c_i < q$. If $q$ is a non-integer, then almost every $x \in J_q$ has continuum many expansions of this form. In this note we consider some
properties of the set $\uu_q$ consisting of numbers $x \in J_q$ having a unique representation of this form. More specifically, we compare the size of the sets $\uu_q$ and $\uu_r$ for values $q$ and $r$ satisfying $1< q < r$ and $m(q)=m(r)$.
\end{abstract}

\maketitle

%\subjclass{Primary 11R06; Secondary 11B83}
%\keywords{}

\section{Introduction}\label{s1}

Beginning with the pioneering works of R\'enyi \cite{R} and Parry \cite{P}, expansions of real numbers in non-integer bases have been widely studied during the last fifty years.

In this paper we consider only sequences of nonnegative integers.
Given a real number $q>1$, an {\it expansion in base} $q$ (or simply {\it expansion}) of a real number $x$ is a sequence $(c_i)=c_1c_2 \ldots$ of integers satisfying

\begin{equation*}
0 \le c_i < q \text{ for each } i \ge 1 \text{ and } x= \sum_{i=1}^{\infty} \frac{c_i}{q^i}.
\end{equation*}
Note that this definition is only meaningful if $x$ belongs to the interval
\begin{equation*}
J_q: = \left[0, \frac{\lceil q \rceil - 1}{q-1} \right] ,
\end{equation*}
where $\lceil q \rceil$ is the smallest integer larger than or equal to $q$. Note that $[0,1] \subset J_q$.

The {\it greedy} expansion of a number $x \in J_q$, denoted by $(b_i(x))$ or $(b_i)$,  can be obtained by performing the greedy algorithm \cite{R}: if for some $n \in \mathbb{N}:=\mathbb{Z}_{\ge 1}$, $b_i = b_i(x)$ is already defined for $1 \le i < n$ (no condition if $n=1$), then $b_n = b_n(x)$ is the largest integer smaller than $q$ such that
\begin{equation*}
\sum_{i=1}^{n} \frac{b_i}{q^i} \le x.
\end{equation*}

If $x \in J_q \setminus \set{0}$, then the {\it quasi-greedy} expansion, denoted by $(a_i(x))$ or $(a_i)$,  is obtained by applying the quasi-greedy algorithm \cite{DK2, KL2, BK}:
if for some $n \in \mathbb{N}$, $a_i = a_i(x)$ is already defined for $1 \le i < n$ (no condition if $n=1$), then $a_n = a_n(x)$ is the largest integer smaller than $q$ such that
\begin{equation*}
\sum_{i=1}^{n} \frac{a_i}{q^i} < x.
\end{equation*}

The quasi-greedy expansion of $x \in J_q \setminus \set{0}$ is always infinite (we call an expansion {\it infinite} if it contains infinitely many nonzero elements; otherwise it is called {\it finite}) and coincides with the greedy expansion $(b_i(x))$ if and only if the latter is infinite. If the greedy expansion of $x \in J_q \setminus \set{0}$ is finite and $b_n$ is its last nonzero element, then
$(a_i(x))= b_1 \ldots b_{n-1}b_n^{-} \alpha_1 \alpha_2 \ldots$, where $b_n^- := b_n - 1$ and $\alpha_i= \alpha_i(q):= a_i(1)$, $i \ge 1$. For
convenience, we set $(a_i(0)):=0^{\infty}$ and refer to it as the quasi-greedy expansion of $0$ in base $q$. We will also write $q \sim (\alpha_i)$ if the quasi-greedy expansion of 1 in base $q$ is given by $(\alpha_i)$.

If $q >1$ is an integer, then the greedy expansion of a number $x \in J_q=[0,1]$ is in fact the only expansion of $x$ in base $q$, except when $x = i/q^n$, where $1 \le i \le q^n - 1$ is an integer and $n \in \mathbb{N}$. However, if $q >1$ is a non-integer, then almost every $x \in J_q$ has continuum many expansions in base $q$, see \cite{DDV}, \cite{S}. Starting with a discovery of Erd\H{o}s, Horv\'ath and Jo\'o \cite{EHJ}, many works during the last fifteen years were devoted to the study of the exceptional set $\uu_q$ consisting of those numbers $x \in J_q$ with a unique expansion in base $q$. For instance, it was shown in \cite{EJK1} that if $1 < q < (1+ \sqrt{5}) /2$, then {\it each} number in the interior of $J_q$ has continuum many expansions. Hence, in this case, $\uu_q = \set{0, 1/(q-1)}$. However, if $q > (1+ \sqrt{5}) /2$, then the set $\uu_q$ is infinite \cite{DK1}.

In order to mention some more sophisticated properties of the set $\uu_q$ for various values of $q$, we introduce the set of {\it univoque numbers} $\uu$, defined by
\begin{equation*}
\uu:=\set{ q >1 : 1 \text{ has a unique expansion in base } q}.
\end{equation*}
It was shown in \cite{EHJ} that the set $\uu \cap (1,2)$ has continuum many elements. Subsequently, the set $\uu$ was characterized lexicographically in \cite{EJK1, EJK2, KL2}, its smallest element $q_1 \approx 1.787$ was determined in \cite{KL1}, and its topological structure was described in \cite{KL2}. It was also shown in \cite{KL1} that the unique expansion of 1 in base $q_1$ is given by the truncated Thue--Morse sequence $(\tau_i)= 11010011 \ldots$, which can be defined recursively by setting $\tau_{2^N}=1$ for $N=0,1,2,\ldots$ and
\begin{equation*}
\tau_{2^N+i} = 1 - \tau_i \quad \mbox{for }
1 \leq i < 2^N, \, N=1,2, \ldots.
\end{equation*}
 Using the structure of this expansion, Glendinning and Sidorov \cite{GS} proved that $\uu_q$ is countable if $1 < q < q_1$ and has the cardinality of the continuum if $q_1 \le q < 2$ (see also \cite{DVK}). They also proved that if $1 < q < q_1$, then the (unique) expansion in base $q$
 of a number $x \in \uu_q$ is ultimately periodic. Finally, the topological structure of the sets $\uu_q$ $(q>1)$ was established in \cite{DVK}.

Let us call a sequence $(c_i)=c_1c_2 \ldots$ with integers $0 \le c_i < q$ {\it univoque in base} $q$ (or simply {\it univoque} if $q$ is understood) if
\begin{equation*}
x=\sum_{i=1}^{\infty} \frac{c_i}{q^i}
\end{equation*}
belongs to $\uu_q$. Let $\uu_q'$ denote the set of all univoque
sequences in base $q$. Clearly, there is a natural bijection
between $\uu_q$ and $\uu_q'$. In what follows we use
systematically the lexicographical order between sequences: \ we
write $(a_i) < (b_i)$ or $a_1a_2 \ldots < b_1b_2 \ldots$ if there
is an integer $n \in \mathbb{N}$ such that $a_i = b_i$ for $i < n$
and $a_n < b_n$.  We recall the following theorem which is essentially due
to Parry \cite{P}:

\begin{theorem}\label{t11}
Let $q>1$ be a real number and let $m$ be the largest integer smaller than $q$.

\begin{itemize}
\item[\rm (i)] A sequence $(b_i)=b_1b_2 \ldots \in \set{0, \ldots, m}^{\mathbb{N}}$ is the greedy expansion of a number $x \in J_q$
if and only if
\begin{equation*}
b_{n+1} b_{n+2} \ldots < \alpha_1 \alpha_2 \ldots \quad \text{whenever} \quad b_n < m.
\end{equation*}
\item[\rm (ii)]  A sequence $(c_i)=c_1c_2 \ldots \in \set{0, \ldots, m}^{\mathbb{N}}$ is univoque if and only if
\begin{equation*}
c_{n+1} c_{n+2} \ldots < \alpha_1 \alpha_2 \ldots \quad \mbox{whenever} \quad c_n < m
\end{equation*}
and
\begin{equation*}
\overline{c_{n+1} c_{n+2}  \ldots} < \alpha_1 \alpha_2 \ldots \quad \mbox{whenever} \quad c_n > 0,
\end{equation*}
where $\overline{c_i}:=m-c_i=\alpha_1-c_i, i \in \mathbb{N}$, and $\overline{c_1 c_2 \ldots} = \overline{c_1} \, \overline{c_2} \ldots$.
\end{itemize}
\end{theorem}

Using the fact that the map $q \mapsto (\alpha_i(q))$ is strictly increasing, it follows at once from this theorem that $\uu_q' \subset \uu_r'$
if $1< q < r$ and $\lceil q \rceil = \lceil r \rceil$. It is the aim of this note to generalize the above mentioned result of
Glendinning and Sidorov~\cite{GS} by considering the difference of the sets $\uu_q'$ and $\uu_r'$, $1 < q < r$, $\lceil q \rceil= \lceil r \rceil$:

\begin{theorem}\label{t12}
Let $1 < q < r$ be real numbers such that $\lceil q \rceil =
\lceil r \rceil$. The following statements are equivalent.

\begin{itemize}
\item [\rm(i)] $(q, r ] \cap \uu = \varnothing$.

\item [\rm(ii)] $(q, r ] \cap \uuu = \varnothing$.

\item [\rm(iii)] Each sequence $(c_i) \in \uu_r' \setminus \uu_q'$ is
ultimately periodic \footnote{We call a sequence $(c_i)=c_1c_2 \ldots$ {\it ultimately periodic} if $c_{n+1} c_{n+2} \ldots$ is periodic for some $n \ge 0$.}.

\item [\rm (iv)] $\uu_r' \setminus \uu_q'$ is countable.

\end{itemize}

\end{theorem}

Incidentally, we will also obtain new characterizations of the set of univoque numbers $\uu$ and its closure $\uuu$ (for other characterizations, see \cite{DVK, KL2}).

\section{Proof of Theorem~\ref{t12}}\label{s2}

Recently, Baiocchi and Komornik \cite{BK} reformulated and extended some classical results of R\'enyi, Parry, Dar\'oczy and K\'atai
\cite{R, P, DK2} by characterizing the quasi-greedy expansions of numbers $x \in J_q \setminus \set{0}$ in a fixed base $q > 1$ (see Proposition~\ref{p22} below).

\begin{proposition}\label{p21}
The map $q \mapsto (\alpha_i(q))$ is a strictly increasing bijection from
the open interval $(1, \infty)$ onto the set of all {\rm infinite} sequences $(\alpha_i)$
satisfying
\begin{equation*}
\alpha_{k+1} \alpha_{k+2} \ldots \leq \alpha_1 \alpha_2 \ldots \quad
\mbox{for all} \quad k\geq 1.
\end{equation*}
\end{proposition}

\begin{proposition}\label{p22}
For each $q>1$, the map $x \mapsto(a_i(x))$ is a strictly increasing bijection
from $(0, \alpha_1 /(q-1)]$ onto the set of all {\rm infinite} sequences $(a_i)$,
satisfying
\begin{equation*}
0 \leq a_n \leq \alpha_1 \quad \mbox{for all} \quad n \geq 1
\end{equation*}
and
\begin{equation*}
a_{n+1}a_{n+2}\ldots \leq \alpha_1 \alpha_2 \ldots \quad \mbox{whenever}
\quad a_n < \alpha_1.
\end{equation*}
\end{proposition}

For any fixed $q > 1$, we introduce the sets
\begin{equation*}
\vv_q:= \set{ x \in J_q : \overline{a_{n+1}(x) a_{n+2} (x) \ldots} \le \alpha_1 (q) \alpha_2 (q) \ldots \quad \text{whenever} \quad a_n(x)>0}
\end{equation*}
and
\begin{equation*}
\vv_q':=\set{(a_i(x)): x \in \vv_q}.
\end{equation*}
It follows from Theorem~\ref{t11} that $\uu_q \subset \vv_q$ for each $q >1$.
Moreover, $\vv_q' \subset \uu_r'$ if $1< q < r$ and $\lceil q \rceil = \lceil r \rceil$.  The precise relationship between the sets $\uu_q$,  its closure $\overline{\uu_q}$ and $\vv_q$ for each $q>1$ was described in \cite{DVK}. For instance, it was shown that $\uu_q$ is closed if and only if $q \notin \uuu$. Moreover, $\uu_q = \overline{\uu_q} = \vv_q$, except when $q$ belongs to the closed null set $\vv$ consisting of those bases $q>1$ such that
$$\overline{\alpha_{k+1}(q) \alpha_{k+2}(q) \ldots} \le \alpha_1(q) \alpha_2(q) \ldots \quad \text{for each} \quad k \ge 1.$$
If $q \in \vv$, then the set $\vv_q \setminus \uu_q$ is countably infinite.

The relationship between the sets $\uu$, $\uuu$, and $\vv$ has been investigated in \cite{KL2}. In particular it was shown that
\begin{itemize}
\item  $\uu \subsetneq \uuu \subsetneq \vv$.
\item $\uuu \setminus \uu$ is countable and dense in $\uuu$.
\item $\vv \setminus \uuu$ is a discrete set, dense in $\vv$.
\item $q \in \uuu$ if and only if $\overline{\alpha_{k+1}(q) \alpha_{k+2}(q) \ldots} < \alpha_1(q) \alpha_2(q) \ldots \quad
\text{for each} \quad k \ge 1.$
\end{itemize}

Applying the above mentioned results, one can easily verify the statements in the following examples.

\begin{examples} \mbox{}
\begin{itemize}

\item
The smallest element of $\vv$ is given by $G:= (1 + \sqrt{5})/2$. Moreover, $G \sim (10)^{\infty}$ and $\vv_G'$ is the set of all sequences in $\set{0,1}^{\mathbb{N}}$ such that a one is never followed by two zeros and a zero is never followed by two ones. Hence $\uu_q$ is infinite if
$G < q \le 2$.
\item
Define the numbers $q_n$ $(n \in \mathbb{N})$ by setting $q_n \sim(110)^{n}(10)^{\infty}$. It follows from Theorem~\ref{t11} that all these numbers belong to $\uu$. However, if we set $q^*:=\lim_{n \to \infty} q_n,$ then $q^* \sim (110)^{\infty}$. Note that $q^* \notin \uu$ because $111(0)^{\infty}$ is another expansion of 1 in base $q^*$. Hence $q^* \in \uuu \setminus \uu$.
\end{itemize}
\end{examples}

Without further comment, we use frequently in the proof below some of the main results in \cite{DVK}, and in particular the
analysis of one of the final remarks at the end of \cite{DVK} which is
concerned with the endpoints of the connected components of $(1,\infty) \setminus \uuu$: if we write $(1,\infty) \setminus \uuu$ as the union of countably many disjoint open intervals (its connected components), then the set $L$ of left endpoints of these intervals is given by $L = \mathbb{N} \cup (\uuu \setminus \uu)$ and the set $R$ of right endpoints of these intervals satisfies the relationship $R \subset \uu$.

\begin{proof}[Proof of Theorem~\ref{t12}]

\rm{ (i) $\Longrightarrow$ (ii)}: Suppose that $(q,r] \cap \uu = \varnothing$. Then $(q, r+ \delta ) \cap \uu = \varnothing$ for some $\delta > 0$ because $\uu$ is closed from above \cite{KL2} and (ii) follows.

\noindent
\rm{ (ii) $\Longrightarrow$ (iii)}: If $(q,r] \cap \uuu = \varnothing$, then $(q, r]$ is a subset of a connected component of $(1, \infty) \setminus \uuu$. Moreover, $[q, r) \cap \vv$ is a finite subset $\set{r_1, \ldots, r_m}$ of $\vv \setminus \uu$, where $r_1 < \cdots < r_m$. Although it is not important in the remainder of the proof, we recall from \cite{DVK} that $r_2, \ldots, r_m \in \vv \setminus \uuu$, but $r_1$ might belong to $\uuu \setminus \uu$. We may write
\begin{equation}\label{21}
\uu_r'  = \uu_q' \cup \bigcup_{\ell =1}^{m} \left( \vv_{r_{\ell}}' \setminus \uu_{r_{\ell}}' \right).
\end{equation}

Fix $\ell \in \set{1, \ldots, m}$ and let $x \in \vv_{r_{\ell}} \setminus \uu_{r_\ell}$. If the greedy expansion $(b_i)$ of $x$ in base $r_{\ell}$ is finite, then $(a_i(x))$ ends with $\alpha_1 \alpha_2 \ldots$. Suppose now that $(b_i)$ is infinite. Since $x \notin \uu_{r_{\ell}}$, there exists an index $n$ such that $b_n > 0$, and $\overline{b_{n+1} b_{n+2} \ldots} \ge \alpha_1 \alpha_2 \ldots$. Since $x \in \vv_{r_{\ell}}$ and $(a_i(x)) = (b_i(x))$, the last inequality is in fact an equality.
Hence the quasi-greedy expansion $(a_i(x))$ of $x$ in base $r_{\ell}$ either ends with $(\alpha_i)$ or $(\overline{\alpha_i})$.
Since $(\alpha_i(q))$ is periodic if $q \in \vv \setminus \uu$ \cite{KL2}, the implication follows from \eqref{21}.

\noindent
\rm{ (iii) $\Longrightarrow$ (iv)} is clear.

\noindent
\rm{ (iv) $\Longrightarrow$ (i)}: We prove the contraposition. Suppose that $(q, r] \cap \uu \not= \varnothing.$ We distinguish between two cases.

If $(q, r) \cap \uu \not= \varnothing,$ then $\vert (q, r) \cap \uuu \vert = 2^{\aleph_0}$ because $\uuu$ is a nonempty perfect set \cite{KL2} and thus each neighborhood of a number $t \in \uuu$ contains uncountably many elements of $\uuu$.
Now
\begin{equation*}
\uu_r'  \setminus  \uu_q' \supset \cup_{t \in (q,r) \cap \uuu} \left( \vv_t' \setminus \uu_t' \right).
\end{equation*}
Hence $\uu_r'  \setminus  \uu_q'$ contains an uncountable union of nonempty disjoint sets and is therefore uncountable.

If $(q, r] \cap \uu = \set{r}$, then $(q,r) \cap \uuu = \varnothing$. Hence by enlarging $q$ if necessary, we may assume that $q \notin \uuu$. Let
\begin{equation*}
\ww_r = \set{ x \in \uu_r : \text{the unique expansion of $x$ in base $r$ belongs to $\uu_q'$}}.
\end{equation*}
We claim that $\ww_r$ is closed. The set $\ww_r$ is a symmetric subset of $J_r$, so it suffices to show that $\ww_r$ is closed from below.
Let $x_i \in \ww_r$ $(i \ge 1)$, and suppose that $x_i \uparrow x$. Let $(c_j^i)$ be the unique expansion of $x_i$ in base $r$, and let
\begin{equation*}
y_i= \sum_{j=1}^{\infty} \frac{c_j^i}{q^j}.
\end{equation*}
Then the increasing sequence $(y_i)$ converges to some $y \in \uu_q$ because $\uu_q$ is a compact set.
Since $(c_j^1) \le  (c_j^2) \le \cdots$, $(c_j^i)$ converges coordinate-wise to the unique expansion $(d_j)$ of $y$ in base $q$ as $i \to \infty$, and \begin{equation*}
x= \sum_{j=1}^{\infty} \frac{d_j}{r^j}.
\end{equation*}
Since $\uu_q' \subset \uu_r'$ we have $x \in \uu_r$, and thus $x \in \ww_r$. Now suppose that $\uu_r' \setminus \uu_q'$ is countable.
Then $\uu_r \setminus \ww_r$ is countable. Note that
$\ww_r \subsetneq \uu_r$ because $\ww_r$  is closed and $\uu_r$ is not. Let $x \in \uu_r \setminus \ww_r$.
Since $\overline{\uu_r} \setminus \uu_r$ is a countable dense subset of $\overline{\uu_r}$ (Theorem 1.3 in \cite{DVK}),
the latter set is perfect, and each neighborhood of $x$ contains uncountably many elements of $\uu_r$ and thus of $\ww_r$.
This contradicts the fact that $\ww_r$ is closed.
\end{proof}

The above result yields new characterizations of $\uu$ and $\uuu$:
\begin{corollary}\label{c23}
A real number $q > 1$ belongs to $\uu$ if and only if $\uu_r' \setminus \uu_q'$ is uncountable for each $r > q$ such that $\lceil q \rceil=\lceil r \rceil$.
\end{corollary}

\begin{proof}
Note that the integers $2, 3, \ldots$ belong to $\uu$. For these values of $q$ the condition in the statement is also vacuously satisfied. Hence we may assume that $q \notin \mathbb{N}$. Suppose that $q \in \uu \setminus \mathbb{N}$. For each $r > q$, $(q, r) \cap \uu \not= \varnothing$ because elements of $\uu \setminus \mathbb{N}$ do not belong to the set of left endpoints of the connected components of $(1, \infty) \setminus \uuu$. Hence
if, in addition, $\lceil q \rceil = \lceil r \rceil$, then $\uu_r' \setminus \uu_q'$ is uncountable by Theorem~\ref{t12}. Conversely, if the latter set is uncountable for each $r > q$ such that $\lceil q \rceil=\lceil r \rceil$, then $(q, r] \cap \uu \not=\varnothing$ for each $r > q$ by Theorem~\ref{t12}, and the result follows because $\uu$ is closed from above.
\end{proof}

For a fixed $r > 1$, let $\mathcal{F}_r' = \cup \uu_q'$, where the union runs over all $q < r$ for which $\lceil q \rceil=\lceil r \rceil$.

\begin{corollary}\label{c24}
Let $r >1$ be a real number. The following statements are equivalent.
\begin{itemize}
\item[\rm (i)] $r \in \uuu$.
\item[\rm (ii)] $\uu_r' \setminus \mathcal{F}_r'$ is uncountable.
\item[\rm (iii)] $\uu_r' \setminus \uu_q'$ is uncountable for each $q < r$ such that $\lceil q \rceil=\lceil r \rceil$.
\item[\rm (iv)] $\uu_r' \setminus \uu_q'$ is nonempty for each $q < r$ such that $\lceil q \rceil=\lceil r \rceil$.
\end{itemize}
\end{corollary}

\begin{proof}
It is clear that $\rm{ (ii) \Longrightarrow (iii) \Longrightarrow (iv)}$. It remains to show that $\rm{ (i) \Longrightarrow (ii)}$ and $\rm{ (iv) \Longrightarrow (i)}$.

$\rm{ (i) \Longrightarrow (ii)}$: Suppose $r \in \uuu$. Let $(q_n)_{n \ge 1}$ be an increasing sequence that converges to $r$,
such that $q_n \notin \uuu$ and $\lceil q_n \rceil=\lceil r \rceil$, $n \in \mathbb{N}$. This can be done, since $\uuu$ is a null set. Let
\begin{equation*}
\ww_r^n=\set{x \in \uu_r : \text{ the unique expansion of $x$ in base $r$ belongs to } \uu_{q_n}'},
\end{equation*}
and
\begin{equation*}
\ww_r= \cup_{n=1}^{\infty} \ww_r^n.
\end{equation*}
It follows from the proof of Theorem~\ref{t12} that $\ww_r^n$ is closed for each $n \in \mathbb{N}$. Moreover,
$\vert \uu_r \setminus \ww_r \vert = \vert \uu_r' \setminus \mathcal{F}_r' \vert$. We know that $\overline{\uu_r} \setminus \uu_r$ is countable.
If $\uu_r \setminus \ww_r$ were countable,
then $\overline{\uu_r}$ would be an $F_{\sigma}$-set:
\begin{equation}\label{22}
\overline{\uu_r}= \cup_{n=1}^{\infty} \ww_r^n \cup (\cup_{x \in \overline{\uu_r} \setminus \ww_r} \set{x}).
\end{equation}
Note that $\overline{\uu_r}$ is a complete metric space. By Baire's theorem, one of the sets on the right-hand side of \eqref{22} has
a nonempty interior. Since $\overline{\uu_r}$ is a perfect set, each singleton belonging to it is not open. Hence one of the sets $\ww_r^n \subset \uu_r$ has an interior point. But this contradicts the fact that $\overline{\uu_r} \setminus \uu_r$ is dense in $\overline{\uu_r}$ (Theorem 1.3 in \cite{DVK}).

$\rm{ (iv) \Longrightarrow (i)}$: We prove the contraposition. Suppose $r \notin \uuu$. We can choose $q \in (1, r)$ close enough to $r$ such that
$[q,r) \cap \vv = \varnothing$. It follows from \eqref{21} that $\uu_q' = \uu_r'$.
\end{proof}

Let $q>1$ be a non-integer, and let $\mathcal{G}_q'= \cap \uu_r'$, where the intersection runs over
all $r >q$ for which $\lceil q \rceil=\lceil r \rceil$. In view of Corollary~\ref{c24} it is natural to ask whether the following variant of Corollary~\ref{c23} holds: the number $q>1$ belongs to $\uu$ if and only if  $\mathcal{G}_q' \setminus \uu_q'$ is uncountable. In order to show that this is {\it not} true, it is sufficient to prove that $\mathcal{G}_q' = \vv_q'$ for each non-integer $q>1$, since  $\vv_q' \setminus \uu_q'$ is known to be countable \cite{DVK}. Let us first recall Lemma 3.2 from \cite{KL2}:

\begin{lemma}\label{l26}
Let $q>1$ be a non-integer, and let $(\beta_i)=\beta_1 \beta_2 \ldots$ be the greedy expansion of $1$ in base $q$. For each $n \in \mathbb{N}$, there exists a number $r=r_n >q$ such that the greedy expansion of $1$ in base $r$ starts with $\beta_1 \ldots \beta_n$.
\end{lemma}

If $q \in \uu \setminus \mathbb{N}$, then 1 has an infinite greedy expansion in base $q$, i.e., $(\alpha_i)=(\alpha_i(q))=(\beta_i(q))$. If a sequence $(a_i) \in \set{0, \ldots, \alpha_1}^{\mathbb{N}}$ belonged to $\mathcal{G}_q' \setminus \vv_q'$, then either there would exist indices $n$ and $m$, such that
$$a_n < \alpha_1 \quad \text{and} \quad a_{n+1} \ldots a_{n+m} > \alpha_1 \ldots \alpha_m$$
or there would exist indices $n$ and $m$, such that
$$a_n > 0 \quad \text{and} \quad \overline{a_{n+1} \ldots a_{n+m}} > \alpha_1 \ldots \alpha_m.$$
If $r_m >q$ is the number that is defined in Lemma~\ref{l26}, then $\alpha_i(q)=\alpha_i(r_m)$ for $1 \le i \le m$, and thus
$(a_i) \notin \uu_{r_m}'$ which is a contradiction.
On the other hand, $\vv_q' \subset \uu_r'$ for each $r>q$ such that $\lceil q \rceil=\lceil r \rceil$, and therefore
$\mathcal{G}_q' = \vv_q'$.

If $q \in (1, \infty) \setminus \uu$, then the equality $\mathcal{G}_q' = \vv_q'$ easily follows from Theorem 1.7 in \cite{DVK}.
\vskip.3cm
\noindent
{\it Acknowledgement.} The author is indebted to Vilmos Komornik for his valuable suggestions and for a careful reading of the manuscript.

\end{document}